\documentclass[12pt,a4paper]{amsart}
\allowdisplaybreaks[1]

\usepackage{amstext}
\usepackage{amsthm}
\usepackage{amssymb}

\newtheorem{thm}{Theorem}[section]
\newtheorem{lem}[thm]{Lemma}

\theoremstyle{definition}
\newtheorem{defn}[thm]{Definition}

\theoremstyle{remark}
\newtheorem{rem}[thm]{Remark}

\begin{document}

\title{An interpolation of Ohno's relation to complex functions}

\author{Minoru Hirose}
\address[Minoru Hirose]{Faculty of Mathematics, Kyushu University 744, Motooka, Nishi-ku,
Fukuoka, 819-0395, Japan}
\email{m-hirose@math.kyushu-u.ac.jp}

\author{Hideki Murahara}
\address[Hideki Murahara]{Nakamura Gakuen University Graduate School, 5-7-1, Befu, Jonan-ku,
Fukuoka, 814-0198, Japan}
\email{hmurahara@nakamura-u.ac.jp}

\author{Tomokazu Onozuka}
\address[Tomokazu Onozuka]{Multiple Zeta Research Center, Kyushu University 744, Motooka, Nishi-ku,
Fukuoka, 819-0395, Japan}
\email{t-onozuka@math.kyushu-u.ac.jp}


\begin{abstract}
 Ohno's relation is a well known formula among multiple zeta values.
 In this paper, we present its interpolation to complex functions. 
\end{abstract}

\maketitle

\section{Introduction}
For complex numbers $s_1,\ldots,s_r \in\mathbb{C}$, we define the multiple zeta function (MZF) by
\begin{align*}
 \zeta(s_{1},\dots,s_{r}):=\sum_{1\le n_{1}<\cdots<n_{r}}\frac{1}{n_{1}^{s_{1}}\cdots n_{r}^{s_{r}}}.
\end{align*}
Matsumoto \cite{Mat02} proved that this series is absolutely convergent
in the domain 
\[
 \{(s_{1},\ldots,s_{r})\in\mathbb{C}^{r}\;|\;\Re(s(l,r))>r-l+1\;(1\leq l\leq r)\},
\]
where $s(l,r):=s_{l}+\cdots+s_{r}$. Akiyama-Egami-Tanigawa
\cite{AET01} and Zhao \cite{Zha00} independently proved that $\zeta(s_{1},\dots,s_{r})$
is meromorphically continued to the whole space $\mathbb{C}^{r}$.
The special values $\zeta(k_{1},\dots,k_{r})$ ($k_{i}\in\mathbb{Z}_{\ge1}\,(i=1,\ldots,r-1),$
$k_{r}\in\mathbb{Z}_{\ge2}$) of MZF are called the multiple zeta values (MZVs). 
The MZVs are real numbers and known to satisfy many kinds
of algebraic relations over $\mathbb{Q}$. 
One of the most well known formulas in this field is Ohno's relation.
We say that an index $(k_1\ldots,k_r)\in\mathbb{Z}_{\ge1}^r$ is admissible if $k_r\ge2$.
\begin{defn}
For an admissible index 
 \[
  \boldsymbol{k}:=(\underbrace{1,\ldots,1}_{a_1-1},b_1+1,\dots,\underbrace{1,\ldots,1}_{a_l-1},b_l+1) \quad (a_p, b_q\ge1),
 \]
we define the dual index of $\boldsymbol{k}$ by 
 \[
  \boldsymbol{k}^\dagger :=(\underbrace{1,\ldots,1}_{b_l-1},a_l+1,\dots,\underbrace{1,\ldots,1}_{b_1-1},a_1+1).
 \]
\end{defn}
\begin{thm}[Ohno's relation; Ohno \cite{Oho99}] \label{ohno}
 For an admissible index $(k_1,\ldots,k_r)\in\mathbb{Z}_{\ge1}^r$ and $m\in\mathbb{Z}_{\ge 0}$, we have
 \begin{align*}
  \sum_{\substack{ e_1+\cdots+e_r=m \\ e_i\ge0\,(1\le i\le r) }}
  \zeta (k_1+e_1,\ldots,k_r+e_r) 
  =\sum_{\substack{ e'_1+\cdots+e'_{r'}=m \\ e'_i\ge0\,(1\le i\le r') }}
  \zeta (k'_1+e'_1,\ldots,k'_{r'}+e'_{r'}), 
 \end{align*}
 where the index $(k'_1,\ldots,k'_{r'})$ is the dual index of $(k_1,\ldots,k_r)$.
\end{thm}

From an analytic point of view, Matsumoto \cite{Mat06} raised the
question whether the known relations among MZVs are valid only
for positive integers or not. 
It is known that the harmonic relations,
e.g., $\zeta(s_{1})\zeta(s_{2})=\zeta(s_{1},s_{2})+\zeta(s_{2},s_{1})+\zeta(s_{1}+s_{2})$
are valid not only for positive integers but for complex numbers.
Unfortunately, there are no known such relations except for the harmonic relations above or relations which contain infinite sums of MZF obtained by the authors (see Hirose-Murahara-Onozuka \cite{HMO18}). 

As a weaker version of the question, we can consider the problem whether the known relaions among MZVs 
can be interpolated by complex variable functions. 
As related to this question, there are several studies using the Mordell-Tornheim zeta functions 
(see e.g., Matsumoto-Tsumura \cite{MT05}). 
Based on such circumstances, we give a complex variable interpolation of Theorem \ref{ohno}. 

For an admissible index $\boldsymbol{k}=(k_1\ldots,k_r)\in\mathbb{Z}_{\ge1}^r$ and $s\in\mathbb{C}$, we define
\[
 I_{\boldsymbol{k}}(s):=\sum_{i=1}^{r}\sum_{0<n_1<\cdots<n_r} \frac{1}{n_1^{k_1}\cdots n_r^{k_r}}
  \cdot \frac{1}{n_i^{s}} \prod_{j\ne i} \frac{n_j}{ n_j-n_i }.
\]
By Zhao-Zhou \cite[Proposition 2.1]{ZhZh11}, we can easily check that this series converges absolutely when $\Re(s)>-1$. 
This is a sum of special cases of $\zeta_{\mathfrak{sl}(r+1)}(\boldsymbol{s})$ which is called the Witten MZF associated with $\mathfrak{sl}(r+1)$. 
This function is first introduced in Matsumoto-Tsumura \cite{MaTs06}, which is also called the zeta function associated with the root system of type $A_r$ (for more details, see Komori-Matsumoto-Tsumura \cite{KMT10}), and continued meromorphically to the whole complex space $\mathbb{C}^{r(r+1)/2}$. 
Hence $I_{\boldsymbol{k}}(s)$ can be continued meromorphically to $\mathbb{C}$.
\begin{thm} \label{main}
 For an admissible index $\boldsymbol{k}$ and $s\in\mathbb{C}$, we have
 \begin{align*}
  I_{\boldsymbol{k}}(s)=I_{\boldsymbol{k}^\dagger}(s).
 \end{align*}
\end{thm}

\begin{rem}
 As we shall see in the next section, Theorem \ref{main} is a generalization of Theorem \ref{ohno} (see Lemma \ref{ikohno}).
\end{rem}

\section{Proof of theorem \ref{main} }
\begin{lem} \label{partial_frac}
 For $m\in\mathbb{Z}_{\ge 0}$ and $a_1,\ldots,a_r\in\mathbb{R}$ with $a_i\neq a_j$ for $i\neq j$, we have
 \begin{align*}  
  \sum_{\substack{ e_1+\cdots+e_r=m \\ e_i\ge0\,(1\le i\le r) }} 
  a_1^{e_1}\cdots a_r^{e_r} 
  =\sum_{i=1}^r a_i^{m+r-1} \prod_{j\ne i}(a_i-a_j)^{-1}. 
 \end{align*} 
\end{lem}
\begin{proof}
 By putting
 \[
  A_i:=a_i^{r-1} \prod_{j\ne i}(a_i-a_j)^{-1},
 \]
 we have
 \[
  \frac{1}{1-a_1x} \cdots \frac{1}{1-a_rx}=\frac{A_1}{1-a_1x}+\cdots+\frac{A_r}{1-a_rx}. 
 \]
 Then we find the desired result. 
\end{proof}
\begin{lem} \label{ikohno}
 For an admissible index $\boldsymbol{k}=(k_1,\ldots,k_r)\in\mathbb{Z}_{\ge1}^r$ and $m\in\mathbb{Z}_{\ge 0}$, we have
 \[
  I_{\boldsymbol{k}}(m)
  =\sum_{\substack{ e_1+\cdots+e_r=m \\ e_i\ge0\,(1\le i\le r) }}
   \zeta (k_1+e_1,\ldots,k_r+e_r).
 \]
\end{lem}
\begin{proof}
 By Lemma \ref{partial_frac}, we have
 \begin{align*}  
  I_{\boldsymbol{k}}(m) 
  &=\sum_{0<n_1<\cdots<n_r} \frac{1}{n_1^{k_1}\cdots n_r^{k_r}}
   \sum_{i=1}^{r} \left(\frac{1}{n_i}\right)^{m+r-1} \prod_{j\ne i} \frac{n_i n_j}{ n_j-n_i } \\
  &=\sum_{0<n_1<\cdots<n_r} \frac{1}{n_1^{k_1}\cdots n_r^{k_r}}
   \sum_{i=1}^r \left(\frac{1}{n_i}\right)^{m+r-1} \prod_{j\ne i} \left(\frac{1}{n_i}-\frac{1}{n_j}\right)^{-1} \\
  &=\sum_{0<n_1<\cdots<n_r} \frac{1}{n_1^{k_1}\cdots n_r^{k_r}}
   \sum_{\substack{ e_1+\cdots+e_r=m \\ e_i\ge0\,(1\le i\le r) }}
   \left(\frac{1}{n_1}\right)^{e_1}\cdots \left(\frac{1}{n_r}\right)^{e_r} \\
  &=\sum_{\substack{ e_1+\cdots+e_r=m \\ e_i\ge0\,(1\le i\le r) }}
   \zeta (k_1+e_1,\ldots,k_r+e_r). \qedhere 
 \end{align*}
\end{proof}

In the proof of Theorem \ref{main}, we use the following lemma (for details, see e.g., Apostol \cite{Apo98}). 
\begin{lem} \label{dirichlet}
 Given two Dirichlet series
 \[
  F(s):=\sum_{n=1}^\infty \frac{f(n)}{n^s} \quad\textrm{and}\quad G(s):=\sum_{n=1}^\infty \frac{g(n)}{n^s},
 \]
 both absolutely convergent for $\Re(s)>\sigma_a$. 
 If $F(s)=G(s)$ for each s in an infinite sequence $\{s_k\}$ such that $\Re(s_k)\rightarrow\infty$, 
 then $f(n)=g(n)$ for every $n$.
\end{lem}
\begin{proof}[Proof of Theorem \ref{main}]
 By Theorem \ref{ohno} and Lemma \ref{ikohno}, we have 
 $I_{\boldsymbol{k}}(s)=I_{\boldsymbol{k}^\dagger}(s)$ for $s\in\mathbb{Z}_{\ge 0}$.
 Since
 \[
  \sum_{n_i=1}^\infty \frac{1}{n_i^s} \left( \sum_{0<n_1<\cdots<n_r} \frac{1}{n_1^{k_1}\cdots n_r^{k_r}} \prod_{j\ne i} \frac{n_j}{n_j-n_i} \right)
 \]
 is a Dirichlet series for each $i$, the function $I_{\boldsymbol{k}}(s)$ is also a Dirichlet series.  
 By Lemma \ref{dirichlet}, we have
 $I_{\boldsymbol{k}}(s)=I_{\boldsymbol{k}^\dagger}(s)$ for $\Re(s)> -1$.
  Since 
 \begin{align*}
  I_{\boldsymbol{k}}(s)
  &=\sum_{i=1}^{r} \sum_{0<n_1<\cdots<n_r} \frac{1}{n_1^{k_1}\cdots n_r^{k_r}}
   \cdot \frac{1}{n_i^{s}} \prod_{j\ne i} \frac{n_j}{ n_j-n_i } \\
  &=\sum_{i=1}^{r} (-1)^{i-1} \sum_{m_1,\ldots,m_r=1}^\infty 
   \frac{1}{ m_1^{k_1-1}(m_1+m_2)^{k_2-1}\cdots(m_1+\cdots+m_r)^{k_r-1} } \\
   &\,\,\,\,\qquad \times \frac{1}{ (m_1+\cdots+m_i)^{s+1} } 
    \, \prod_{j<i} \frac{1}{m_{j+1}+\cdots+m_i }
    \, \prod_{j>i} \frac{1}{m_{i+1}+\cdots+m_j },
 \end{align*}
 $I_{\boldsymbol{k}}(s)$ can be regarded as the sum of the zeta functions associated with the root system of type $A_r$.
 Thus $I_{\boldsymbol{k}}(s)$ can be meromorphically continued to the whole space of $\mathbb{C}$.
\end{proof}

\section*{Acknowledgements}
The authors would like to thank Professor Shingo Saito, Doctor Nobuo Sato and Doctor Shin-ya Kadota for valuable comments. 
This work was supported by JSPS KAKENHI Grant Numbers JP18J00982, JP18K13392.


\end{document}